\documentclass{amsart}
\usepackage{setspace}
\usepackage{a4}
\usepackage{amsthm}
\usepackage{latexsym}
\usepackage{amsfonts}
\usepackage{graphicx}
\usepackage{textcomp}
\usepackage{cite}
\usepackage{enumerate}
\usepackage{amssymb}
\usepackage{hyperref}
\usepackage{amsmath}
\usepackage{tikz}
\usepackage[mathscr]{euscript}
\usepackage{mathtools}
\newtheorem{theorem}{Theorem}[section]

\newtheorem{example}[theorem]{Example}

\title{This is the title}
\usepackage{fancyhdr}

\begin{document}
\hrule\hrule\hrule\hrule\hrule
\vspace{0.3cm}	
\begin{center}
{\bf\large{{Non-Archimedean Tarski-Maligranda Inequalities}}}\\
\vspace{0.3cm}
\hrule\hrule\hrule\hrule\hrule
\vspace{0.3cm}
\textbf{K. Mahesh Krishna}\\
School of Mathematics and Natural Sciences\\
Chanakya University Global Campus\\
NH-648, Haraluru Village\\
Devanahalli Taluk, 	Bengaluru  North District\\
Karnataka State, 562 110, India\\
Email: kmaheshak@gmail.com\\

Date: \today
\end{center}
\hrule\hrule
\vspace{0.5cm}
\textbf{Abstract}: In 1930, Tarski observed that 
\begin{align*}
	\bigg||r|-|s|\bigg|=|r-s|+ |r+s|-(|r|+|s|), \quad \forall r, s \in \mathbb{R}.
\end{align*}
 In 2008, Maligranda converted the previous equality into inequalities that are valid in every normed linear space. We derive non-Archimedean versions of Tarski-Maligranda inequalities. Difference between Archimedean and non-Archimedean inequalities is surprising.

\hrule

\hrule
\section{Introduction}
It is universally known that 
\begin{align}\label{TI}
	\bigg||r|-|s|\bigg|\leq \min\{|r-s|, |r+s|\}, \quad \forall r, s \in \mathbb{R}.
\end{align}
In 1930, Tarski noticed a strengthening of Inequality (\ref{TI}) (Page 688, \cite{TARSKI}):
\begin{align}\label{TE}
		\bigg||r|-|s|\bigg|=|r-s|+ |r+s|-(|r|+|s|), \quad \forall r, s \in \mathbb{R}.
\end{align}
It is easy to see that Equality (\ref{TE}) fails for complex numbers. In 2008, Maligranda showed that one can generalize Equality (\ref{TE}) by deriving  inequalities that hold for any normed linear space (NLS) \cite{MALIGRANDA}.
\begin{theorem} \cite{MALIGRANDA} \label{MT} (\textbf{Tarski-Maligranda Inequalities})
Let $\mathcal{X}$ be a NLS. Then for all $x, y \in \mathcal{X}\setminus\{0\}$, 
\begin{align}\label{MI1}
	\bigg|\|x\|-\|y\|\bigg|&\leq \|x+y\|+\|x-y\|-(\|x\|+\|y\|)\leq \min\{\|x-y\|, \|x+y\|\}
\end{align}
and 
\begin{align}\label{MI2}
	\bigg|\|x\|-\|y\|\bigg|\leq \|x\|+\|y\|-\bigg|\|x+y\|-\|x-y\|\bigg|.
\end{align}
\end{theorem}
We naturally seek non-Archimedean versions of Inequalities (\ref{MI1}) and (\ref{MI2}). In this note, we derive them.

\section{Non-Archimedean Tarski-Maligranda Inequalities}

Let $\mathbb{K}$ be a field. A map $|\cdot|: \mathbb{K} \to [0, \infty)$ is said to be a non-Archimedean  valuation  if following conditions hold.
\begin{enumerate}[\upshape(i)]
	\item If $\lambda  \in \mathbb{K}$ is such that $|\lambda|=0$, then $\lambda=0$.
	\item $|\lambda \mu|=|\lambda||\mu|$ for all $\lambda, \mu  \in \mathbb{K}$.
	\item (Ultra-triangle inequality) $|\lambda+\mu|\leq \max\{|\lambda|, |\mu|\}$ for all $\lambda, \mu \in \mathbb{K}$.
\end{enumerate}
In this case, $\mathbb{K}$ is called as non-Archimedean valued field \cite{SCHIKHOF}.  Let $\mathcal{X}$ be a vector space over a non-Archimedean valued field $\mathbb{K}$ with valuation $|\cdot|$. A map $\|\cdot\|: \mathcal{X} \to [0, \infty)$ is said to be a non-Archimedean norm if following conditions hold.
\begin{enumerate}[\upshape(i)]
	\item If $x \in \mathcal{X}$ is such that $\|x\|=0$, then $x=0$.
	\item $\|\lambda x\|=|\lambda|\|x\|$ for all $\lambda \in \mathbb{K}$, for all $x \in \mathcal{X}$.
	\item (Ultra-norm inequality) $\|x+y\|\leq \max\{\|x\|, \|y\|\}$ for all $x,y \in \mathcal{X}$.
\end{enumerate}
In this case, $\mathcal{X}$ is called as non-Archimedean linear space (NALS) \cite{GARCIASCHIKHOF}.   Following is non-Archimedean version of Theorem \ref{MT}.
\begin{theorem}\label{NTM}
(\textbf{Non-Archimedean Tarski-Maligranda Inequalities})
Let $\mathcal{X}$ be a NALS over a non-Archimedean valued field $\mathbb{K}$. Assume that the characteristic of the field $\mathbb{K}$ is not 2. Then for all $x, y \in \mathcal{X}$, 
\begin{align*}
	\bigg|\|x\|-\|y\|\bigg|&\leq \frac{2}{|2|}\max\{\|x-y\|, \|x+y\|\}-(\|x\|+\|y\|)\\
	&\leq \frac{2}{|2|}\max\{\|x\|, \|y\|\}-(\|x\|+\|y\|)
\end{align*}
and 
\begin{align*}
	\bigg|\|x\|-\|y\|\bigg|\leq \|x\|+\|y\|-\frac{2}{|2|}\bigg|\|x+y\|-\|x-y\|\bigg|.
\end{align*}	
In particular, if $|2|=1$, then 
\begin{align*}
\bigg|\|x\|-\|y\|\bigg|&\leq 2\max\{\|x-y\|, \|x+y\|\}-(\|x\|+\|y\|)\\
&\leq 2\max\{\|x\|, \|y\|\}-(\|x\|+\|y\|)
\end{align*}
and 
\begin{align*}
\bigg|\|x\|-\|y\|\bigg|\leq \|x\|+\|y\|-2\bigg|\|x+y\|-\|x-y\|\bigg|.
\end{align*}	
\end{theorem}
\begin{proof}
	Let $x, y \in \mathcal{X}$. Then 
	\begin{align}\label{F1}
		\|x\|+\|y\|+\bigg|\|x\|-\|y\|\bigg|=2\max\{\|x\|, \|y\|\}.
	\end{align}	
We also have 
\begin{align}\label{A1}
	\max\{\|x-y\|, \|x+y\|\}\geq \|(x+y)+(x-y)\|=\|2x\|=|2|\|x\|
\end{align}
and 
\begin{align}\label{A2}
	\max\{\|x-y\|, \|x+y\|\}\geq \|(x+y)-(x-y)\|=\|2y\|=|2|\|y\|.
\end{align}
Inequalities (\ref{A1}) and (\ref{A2}) give 
\begin{align}\label{F2}
		\max\{\|x-y\|, \|x+y\|\}\geq |2|\max\{\|x\|, \|y\|\}.
\end{align}
Inequalities (\ref{F1}) and (\ref{F2}) give 
\begin{align*}
	\|x\|+\|y\|+\bigg|\|x\|-\|y\|\bigg|=2\max\{\|x\|, \|y\|\}\leq \frac{2}{|2|}\max\{\|x-y\|, \|x+y\|\}.
\end{align*}
Rearranging, 
\begin{align*}
\bigg|\|x\|-\|y\|\bigg|\leq \frac{2}{|2|}\max\{\|x-y\|, \|x+y\|\}-(\|x\|+\|y\|).
\end{align*}
We next see that 
	\begin{align}\label{E1}
	\|x\|+\|y\|-\bigg|\|x\|-\|y\|\bigg|=2\min\{\|x\|, \|y\|\}.
\end{align}	
Further, 
\begin{align}\label{G1}
	\bigg|\|x+y\|-\|x-y\|\bigg|\leq	\|(x+y)+(x-y)\|=\|2x\|=|2|\|x\|
\end{align}
and 
\begin{align}\label{G2}
	\bigg|\|x+y\|-\|x-y\|\bigg|\leq	\|(x+y)-(x-y)\|=\|2y\|=|2|\|y\|.
\end{align}
Inequalities (\ref{G1}) and (\ref{G2}) give 
\begin{align}\label{E2}
	\bigg|\|x+y\|-\|x-y\|\bigg|\leq	|2|\min\{\|x\|, \|y\|\}.
\end{align}
Inequalities (\ref{E1}) and (\ref{E2}) give 
\begin{align*}
		\bigg|\|x+y\|-\|x-y\|\bigg|\leq	|2|\min\{\|x\|, \|y\|\}= \frac{|2|}{2}\left(\|x\|+\|y\|-\bigg|\|x\|-\|y\|\bigg|\right).
\end{align*}
Rearranging, 
\begin{align*}
\bigg|\|x\|-\|y\|\bigg|\leq \|x\|+\|y\|-\frac{2}{|2|}\bigg|\|x+y\|-\|x-y\|\bigg|.
\end{align*}
\end{proof}
\begin{example}
	Let $\mathcal{X}$ be a NALS over $\mathbb{Q}_2$. We then have $|2|_2=1/2$. Theorem \ref{NTM} gives 
	\begin{align*}
		\bigg|\|x\|-\|y\|\bigg|&\leq \frac{2}{|2|_2}\max\{\|x-y\|, \|x+y\|\}-(\|x\|+\|y\|)\\
		&=4\max\{\|x-y\|, \|x+y\|\}-(\|x\|+\|y\|), \quad \forall x, y \in \mathcal{X}
	\end{align*}
and 
\begin{align*}
	\bigg|\|x\|-\|y\|\bigg|&\leq \|x\|+\|y\|-\frac{2}{|2|_2}\bigg|\|x+y\|-\|x-y\|\bigg|\\
	&=\|x\|+\|y\|-4\bigg|\|x+y\|-\|x-y\|\bigg|, \quad \forall x, y \in \mathcal{X}.
\end{align*}	
\end{example}
\begin{example}
	Let $p>2$ be a prime number. 	Let $\mathcal{X}$ be a NALS over $\mathbb{Q}_p$. We then have $|2|_p=1$. Theorem \ref{NTM} gives 
	\begin{align*}
		\bigg|\|x\|-\|y\|\bigg|&\leq \frac{2}{|2|_p}\max\{\|x-y\|, \|x+y\|\}-(\|x\|+\|y\|)\\
		&=2\max\{\|x-y\|, \|x+y\|\}-(\|x\|+\|y\|), \quad \forall x, y \in \mathcal{X}
	\end{align*}
	and 
	\begin{align*}
		\bigg|\|x\|-\|y\|\bigg|&\leq \|x\|+\|y\|-\frac{2}{|2|_p}\bigg|\|x+y\|-\|x-y\|\bigg|\\
		&=\|x\|+\|y\|-2\bigg|\|x+y\|-\|x-y\|\bigg|, \quad \forall x, y \in \mathcal{X}.
	\end{align*}	
\end{example}

\section{Conclusions}
\begin{enumerate}
	\item In 1930, Tarski observed an equality in triangle inequality for real numbers.
	\item In 2008, Maligranda converted equality by Tarski to  inequalities which are valid in every normed linear spaces.
\item In this note, we  derived non-Archimedean version of Tarski-Maligranda inequalities. 
\end{enumerate}

\section{Acknowledgments:} The author thanks the anonymous reviewer for his/her  suggestions, which improved the article.
 \bibliographystyle{plain}
 \bibliography{reference.bib}

\end{document}